\documentclass[a4paper,11pt]{article}
\usepackage{amsmath, amsfonts, amscd, amssymb, amsthm, enumerate, xypic, psfrag, mathdots}

\def\bfB{\mathbf{B}}

\DeclareMathOperator{\Mat}{\operatorname{M}}
\DeclareMathOperator{\GL}{\operatorname{GL}}
\DeclareMathOperator{\Ortho}{\operatorname{O}}
\DeclareMathOperator{\id}{\operatorname{id}}
\DeclareMathOperator{\Ker}{\operatorname{Ker}}
\DeclareMathOperator{\im}{\operatorname{Im}}
\DeclareMathOperator{\Diag}{\operatorname{Diag}}
\DeclareMathOperator{\Vect}{\operatorname{span}}
\DeclareMathOperator{\car}{\operatorname{char}}
\DeclareMathOperator{\tr}{\operatorname{tr}}
\DeclareMathOperator{\Tr}{\operatorname{Tr}}
\DeclareMathOperator{\Gal}{\operatorname{Gal}}
\DeclareMathOperator{\rk}{\operatorname{rk}}

\DeclareMathOperator{\codim}{\operatorname{codim}}

\renewcommand{\setminus}{\smallsetminus}

\def\K{\mathbb{K}}

\def\C{\mathbb{C}}

\def\N{\mathbb{N}}
\def\Z{\mathbb{Z}}

\renewcommand{\L}{\mathbb{L}}

\def\calP{\mathcal{P}}

\def\lcro{\mathopen{[\![}}
\def\rcro{\mathclose{]\!]}}

\theoremstyle{definition}
\newtheorem{Def}{Definition}

\theoremstyle{plain}
\newtheorem{theo}{Theorem}
\newtheorem{prop}[theo]{Proposition}
\newtheorem{cor}[theo]{Corollary}
\newtheorem{lemme}[theo]{Lemma}
\newtheorem{claim}{Claim}

\theoremstyle{plain}

\theoremstyle{remark}
\newtheorem{Rems}{Remarks}
\newtheorem{Rem}[Rems]{Remark}

\title{When does a linear map belong to at least one orthogonal or symplectic group?}
\author{Cl\'ement de Seguins Pazzis\footnote{Lyc\'ee Priv\'e Sainte-Genevi\`eve, 2, rue
de l'\'Ecole des Postes, 78029 Versailles Cedex, FRANCE.}
\footnote{e-mail address: dsp.prof@gmail.com}}

\begin{document}

\thispagestyle{plain}
\maketitle

\begin{abstract}
Given an endomorphism $u$ of a finite-dimensional vector space over an arbitrary field $\K$, we give necessary and sufficient conditions
for the existence of a regular quadratic form (respectively, a symplectic form) for which $u$ is orthogonal
(respectively, symplectic). Since a solution to this problem is already known in the case $\car(\K) \neq 2$,
our main contribution lies in the case $\car(\K)=2$. When $\car(\K)=2$, we also give necessary and sufficient conditions
for the existence of a regular symmetric bilinear form for which $u$ is orthogonal.
When $\K$ is finite with characteristic $2$, we give necessary and sufficient conditions for the existence
of an hyperbolic quadratic form (respectively, a regular non-hyperbolic quadratic form, respectively, a regular nonalternate symmetric bilinear form)
for which $u$ is orthogonal.
\end{abstract}



\section{Introduction}

\subsection{The problem}

In this paper, $\K$ denotes an arbitrary field and $\car(\K)$ is its characteristic.
We choose an algebraic closure $\overline{\K}$ of $\K$.
We use the French convention for integers: $\N$ denotes the set of non-negative integers, and $\N^*:=\N \setminus\{0\}$ the set of positive ones.
Given integers $a$ and $b$, we denote by $\lcro a,b\rcro$ the set of all \emph{integers} $n$ such that $a \leq n \leq b$.

We denote by $\Mat_n(\K)$ the algebra of square matrices with $n$ rows and entries in $\K$.
A matrix $A$ of $\Mat_n(\K)$ is called \textbf{alternate} when it is skew-symmetric with zero diagonal entries,
i.e., $\forall X \in \K^n, \; X^TAX=0$ (of course, when $\car(\K) \neq 2$, the alternate matrices of $\Mat_n(\K)$
are its skew-symmetric matrices). \\
Similarity of two matrices $A$ and $B$ is written $A \sim B$.
For $k \in \N^*$ and $\lambda \in \K$, we denote by $J_k(\lambda) \in \Mat_k(\K)$ the Jordan matrix
of order $k$ with eigenvalue $\lambda$.

Given a monic polynomial $P=x^n-\underset{k=0}{\overset{n-1}{\sum}}a_kx^k \in \K[x]$, we denote by
$$C(P):=\begin{bmatrix}
0 &   & & 0 & a_0 \\
1 & 0 & &   & a_1 \\
0 & \ddots & \ddots & & \vdots \\
\vdots & & & 0 & a_{n-2} \\
0 & & &  1 & a_{n-1}
\end{bmatrix}$$
its \textbf{companion matrix}.

Given a vector space $V$ over $\K$, we denote by $\GL(V)$ the group of linear automorphisms of $V$.

A \textbf{symplectic form} is a non-degenerate alternate form. For such a form $b$ on a vector space $V$,
a \textbf{symplectic morphism} of $(V,b)$ is an automorphism $u$ of $V$ such that
$\forall (x,y)\in V^2, \; b(u(x),u(y))=b(x,y)$.

\vskip 3mm The reduction of special endomorphisms (e.g.\ selfadjoint endomorphisms, orthogonal - or unitary - endomorphisms,
normal endomorphisms) plays an important part in the study of real and complex vector spaces.
The generalization to an arbitrary quadratic or symplectic form, however, is much more difficult
(see the early treatments in \cite{Cushman}, \cite{Cikunov}, \cite{Milnor}, \cite{Williamson1}, \cite{Williamson2}).
A complete classification of selfadjoint, skew-selfadjoint and orthogonal automorphisms
is however known up to the classification of hermitian forms when $\car(\K) \neq 2$ (see \cite{Sergei1} and \cite{Sergei2})

Instead of trying to find canonical forms for special morphisms in this setting, an easier problem
is to find necessary and sufficient conditions for an endomorphism to be selfadjoint, skew-selfadjoint or orthogonal
for at least one regular quadratic form (or for a symplectic form).
The first important result on the topic was obtained by Frobenius \cite{Frobenius}, who proved
that every endomorphism of a finite-dimensional vector space $V$
is selfadjoint for at least one regular symmetric bilinear form on $V$
(to put things differently, every square matrix is the product of two symmetric matrices, one of which is nonsingular).
Later, Stenzel \cite{Stenzel} determined when an endomorphism could be skew-selfadjoint for a regular quadratic form,
or selfadjoint or skew-selfadjoint for a symplectic form: he only tackled the case of complex vector spaces
but his results were later generalized to an arbitrary field \cite{GowLaffey}.

In this paper, we tackle the case of the automorphisms of a finite-dimensional vector space that are
orthogonal (respectively, symplectic) for at least one regular quadratic form (respectively, symplectic form).
In chapter XI of \cite{Gantmacher} and more recently in \cite{HornMerino},
this problem is solved for orthogonal morphisms when the underlying field is $\C$, but the proof generalizes to an arbitrary
algebraically closed field of characteristic not $2$
(this yields the possible Jordan canonical forms of the matrices in the orthogonal group $\Ortho_n(\C)$).
The solution for symplectic morphisms is also known \cite{Gongo} for algebraically closed fields of characteristic not $2$.
The deep results from \cite{Sergei1} yield the solution to both problems for an arbitrary field of characteristic not $2$.

Here, we completely solve the problem for an arbitrary field, possibly of characteristic $2$. Although
the results are already known in the case $\car(\K) \neq 2$, we reprove them along the way
since doing so has a low additional cost.

\begin{Def}
Let $u \in \GL(V)$ for some finite-dimensional vector space $V$ over $\K$.
We say that $u$ is:
\begin{itemize}
\item \textbf{essentially orthogonal} when $u$ is $q$-orthogonal for some regular quadratic form $q$
on $V$, i.e., $\forall x \in V, \; q(u(x))=q(x)$;
\item \textbf{essentially bilin-orthogonal}
when $u$ is an isometry for some regular symmetric bilinear form $b$ on $V$,
i.e., $\forall (x,y)\in V^2, \; b(u(x),u(y))=b(x,y)$;
\item \textbf{essentially symplectic} when $u$ is a symplectic morphism for some
symplectic form $b$ on $V$, i.e., $\forall (x,y)\in V^2, \; b(u(x),u(y))=b(x,y)$.
\end{itemize}
\end{Def}

When $\car(\K) \neq 2$, the essentially orthogonal morphisms are
the essentially bilin-orthogonal ones. When $\car(\K)=2$, the following
implications hold:
$$\text{$u$ essentially orthogonal} \; \Rightarrow \;
\text{$u$ essentially symplectic} \; \Rightarrow \;
\text{$u$ essentially bilin-orthogonal.}$$
Indeed, the polar form of a regular quadratic form is symplectic.

Our main problem may now be stated: determine canonical forms for essentially orthogonal,
essentially bilin-orthogonal, and essentially symplectic morphisms.

We adapt the same definitions to square matrices:
notice then that the set of essentially orthogonal (respectively, essentially bilin-orthogonal, respectively,
essentially symplectic) matrices of $\Mat_n(\K)$ is invariant under similarity,
and we have the following characterizations:
\begin{itemize}
\item a matrix $M \in \GL_n(\K)$ is essentially symplectic if and only if $M^TAM=A$ for some nonsingular alternate matrix $A$;
\item a matrix $M \in \GL_n(\K)$ is essentially bilin-orthogonal if and only if $M^TSM=S$ for some nonsingular symmetric matrix $S$;
\item if $\car(\K)=2$, then $M \in \GL_n(\K)$ is essentially orthogonal if and only if
there exists $A \in \Mat_n(\K)$ such that $M^TAM+A$ is alternate and the alternate matrix $A+A^T$ is nonsingular.
\end{itemize}

Let $M \in \GL_n(\K)$ have one of the above properties.
Then there exists a nonsingular matrix $A \in \GL_n(\K)$ such that $M=A^{-1}(M^{-1})^TA$,
hence $M$ must be similar to $(M^{-1})^T$. However, $(M^{-1})^T$
is itself similar to $M^{-1}$ (see \cite{Taussky}).

\begin{prop}
Let $M \in \GL_n(\K)$ which is either essentially orthogonal, essentially symplectic
or essentially bilin-orthogonal. Then $M \sim M^{-1}$.
\end{prop}

As we shall see, the converse is not true (this is obvious for essentially symplectic
morphisms since symplectic forms exist only in even dimensions).

\subsection{Main results}

Before stating our main theorems, we need a few extra definitions:

\begin{Def}
A polynomial $P(x) \in \K[x]$ has \textbf{valuation $0$} if $P(0) \neq 0$. \\
Given a monic polynomial $P=\underset{k=0}{\overset{n}{\sum}} a_k x^k \in \K[x]$ of degree $n$ and valuation $0$, we define
$P^\#:=\frac{1}{a_0}\,\underset{k=0}{\overset{n}{\sum}} a_{n-k} x^k$ and call it the \textbf{reciprocal polynomial} of $P$. \\
We say that $P$ is a \textbf{palindromial} when $P=P^\#$.
\end{Def}

Remark that when $P=\underset{k=1}{\overset{n}{\prod}}(x-\lambda_k)$, one has $P^\#=\underset{k=1}{\overset{n}{\prod}}\bigl(x-\frac{1}{\lambda_k}\bigr)$.
Moreover, the map $P \mapsto P^\#$ defines an involution on the set
of monic polynomials of $\K[x]$ with valuation $0$ and satisfies $(PQ)^\#=P^\#Q^\#$
for all such polynomials: in particular, it preserves divisibility and irreducibility.

\paragraph{}We now state our main results.

\begin{theo}\label{symplectictheorem}
Let $A \in \GL_n(\K)$.
The following conditions are equivalent:
\begin{enumerate}[(i)]
\item $A$ is essentially symplectic.
\item $A$ is similar to $A^{-1}$ and, for every $k \in \N$ and each one of the eigenvalues $1$ and $-1$, the number of Jordan blocks of size
$2k+1$ associated to $A$ is even.
\item $\forall \lambda \in \overline{\K}\setminus\{0,1,-1\}, \; \forall k \in \N^*, \; \rk(A-\lambda I_n)^k=\rk\bigl(A-\frac{1}{\lambda} I_n\bigr)^k$ and, for every $k \in \N$ and each one of the eigenvalues $1$ and $-1$, the number of Jordan blocks of size $2k+1$ associated to $A$ is even.
\item All the elementary factors of $A$ are palindromials
 and, for every $k \in \N$ and each one of the eigenvalues $1$ and $-1$, the number of Jordan blocks of size $2k+1$ associated to $A$ is even.
\item There are nonsingular matrices $B$ and $C$ such that
$A \sim B \oplus B^{-1} \oplus C$, all the elementary factors of $C$
are palindromials and $C$ contains only even-sized Jordan blocks
for the eigenvalues $1$ and $-1$.
\end{enumerate}
\end{theo}

\begin{theo}\label{bilinorthotheorem}
Let $A \in \GL_n(\K)$ and assume that $\car(\K)=2$.
The following conditions are equivalent:
\begin{enumerate}[(i)]
\item $A$ is essentially bilin-orthogonal.
\item $A$ is similar to $A^{-1}$ and, for every $k \in \N^*$,
the number of Jordan blocks of size $2k+1$ associated to $A$ for the eigenvalue $1$ is even.
\item $\forall \lambda \in \overline{\K}\setminus\{0,1\}, \; \forall k \in \N^*, \; \rk(A-\lambda I_n)^k=\rk\bigl(A-\frac{1}{\lambda} I_n\bigr)^k$ and,
for every $k \in \N^*$, the number of Jordan blocks of size $2k+1$ associated to $A$ for
the eigenvalue $1$ is even.
\item All the elementary factors of $A$ are palindromials and,
for every $k \in \N^*$, the number of Jordan blocks of size $2k+1$ associated to $A$ for the eigenvalue $1$ is even.
\item There are nonsingular matrices $B$ and $C$ such that
$A \sim B \oplus B^{-1} \oplus C$, all the elementary factors of $C$
are palindromials and each Jordan block of $C$
for the eigenvalue $1$ is either even-sized or has size $1$.
\end{enumerate}
\end{theo}

\begin{theo}\label{orthotheoremcar<>2}
Let $A \in \GL_n(\K)$ and assume that $\car(\K) \neq 2$.
The following conditions are equivalent:
\begin{enumerate}[(i)]
\item $A$ is essentially orthogonal.
\item $A$ is similar to $A^{-1}$ and, for every $k \in \N^*$ and each one of the eigenvalues $1$ and $-1$, the number of Jordan blocks of size $2k$ associated to $A$ is even.
\item $\forall \lambda \in \overline{\K}\setminus \{0,1,-1\}, \; \forall k \in \N^*, \; \rk(A-\lambda I_n)^k=\rk\bigl(A-\frac{1}{\lambda} I_n\bigr)^k$ and,
for every $k \in \N^*$ and each one of the eigenvalues $1$ and $-1$, the number of Jordan blocks of size $2k$ associated to $A$ is even.
\item All the elementary factors of $A$ are palindromials and,
for every $k \in \N^*$ and each one of the eigenvalues $1$ and $-1$, the number of Jordan blocks of size $2k$ associated to $A$ is even.
\item There are nonsingular matrices $B$ and $C$ such that
$A \sim B \oplus B^{-1} \oplus C$, all the elementary factors of $C$
are palindromials and $C$ contains only odd-sized Jordan blocks
for the eigenvalues $1$ and $-1$.
\end{enumerate}
\end{theo}

\begin{theo}\label{orthotheoremcar2}
If $\car(\K)=2$, then the essentially orthogonal matrices of $\Mat_n(\K)$
are its essentially symplectic ones.
\end{theo}

When $n$ is even, condition (iii) in Theorem \ref{bilinorthotheorem} implies that
the number of Jordan blocks of size $1$ for $A$ is even: this shows that
being essentially bilin-orthogonal is the same as being essentially symplectic
(whereas not every non-degenerate symmetric bilinear form is symplectic).

\paragraph{Structure of the paper:}

In Section \ref{reductions}, we reduce the proofs of Theorems \ref{symplectictheorem} to \ref{orthotheoremcar2} to the following elementary cases:
\begin{itemize}
\item $A$ is similar to $B \oplus B^{-1}$ for some nonsingular matrix $B$ (this is easily
dealt with in Section \ref{BplusB-1});
\item $A$ is \textbf{unipotent} i.e., triangularizable with $1$ as its sole eigenvalue:
see Section \ref{decorsection} for the necessary condition for $A$ to be essentially orthogonal
(respectively, symplectic, respectively, bilin-orthogonal), and Section \ref{Jordanblock} for the sufficient condition;
\item $A$ is the companion matrix of $P^a$ for some integer $a \geq 1$
and some irreducible palindromial $P$ of degree greater than $1$
(see Section \ref{companionsection} for the fact that such a matrix is always essentially orthogonal and essentially
symplectic, hence also essentially bilin-orthogonal). When $\K$ is finite, this involves field extensions and hermitian forms.
\end{itemize}

The last two sections deal with refinements of the above theorems for specific fields of characteristic $2$.
In Section \ref{car2bilinSec}, we determine, when  $\K$ is perfect and $\car(\K)=2$,
the bilin-orthogonal automorphisms $u$ that are orthogonal for at least one nonalternate regular symmetric bilinear form,
i.e., we determine the matrices that are similar to a matrix of the orthogonal group $\text{O}_n(\K)$.
In Section \ref{car2quadSec}, we investigate the essentially orthogonal morphisms when $\K$ is finite and has characteristic $2$.
In that case, there are exactly two equivalence classes of regular quadratic forms on a given
even-dimensional vector space $V$ over $\K$ (namely, the hyperbolic and the non-hyperbolic ones),
and we give necessary and sufficient conditions for an automorphism of $V$ to be orthogonal
for at least one regular quadratic form belonging to a given equivalence class.

\section{Reducing the problem to more elementary ones}\label{reductions}

\subsection{Two basic principles}\label{basicred}

Let $A$ and $B$ be two essentially orthogonal (respectively, essentially symplectic, respectively,
essentially bilin-orthogonal) matrices. Since the orthogonal direct sum of two regular
quadratic forms (respectively, symplectic forms, respectively, symmetric bilinear forms)
is regular, the matrix $A \oplus B:=\begin{bmatrix}
A & 0 \\
0 & B
\end{bmatrix}
$ is essentially orthogonal
(respectively, essentially symplectic, respectively, essentially bilin-orthogonal).

Notice also that if $A$ is an essentially orthogonal (respectively, essentially symplectic,
respectively, essentially bilin-orthogonal) matrix, then its opposite matrix $-A$ is also
essentially orthogonal (respectively, essentially symplectic, respectively,
essentially bilin-orthogonal).

\subsection{When is a nonsingular matrix similar to its inverse?}\label{Asyminverse}

The following characterizations are known but we prove them
as they are crucial to our study:

\begin{prop}
Let $A \in \GL_n(\K)$. The following conditions are then equivalent:
\begin{enumerate}[(i)]
\item $A$ is similar to $A^{-1}$.
\item $\forall \lambda \in \overline{\K} \setminus \{0\}, \; \forall k \in \N^*, \; \rk(A-\lambda I_n)^k=\rk\bigl(A-\frac{1}{\lambda} I_n\bigr)^k$.
\item The elementary factors of $A$ are all palindromials.
\item There are nonsingular matrices $B$ and $C$ such that
$A \sim B \oplus B^{-1} \oplus C$ and all the irreducible monic factors in the minimal polynomial
of $C$ are palindromials.
\end{enumerate}
\end{prop}

\begin{proof}
\begin{itemize}
\item The equivalence between (i) and (ii) is straightforward since
$\rk(A^{-1}-\lambda I_n)^k=\rk A^{-k}\bigl(A-\frac{1}{\lambda} I_n\bigr)^k=\rk\bigl(A-\frac{1}{\lambda} I_n\bigr)^k$
for every $\lambda \in \overline{\K}\setminus\{0\}$ and $k \in \N^*$.
\item Given a monic polynomial $P \in \K[x]$ with valuation $0$,
notice that the companion matrix $C(P)$ has a cyclic inverse with minimal polynomial
$P^\#$, and hence is similar to $C(P^\#)$. As $P \mapsto P^\#$ preserves divisibility,
if the elementary factors of $A$ are $P_1,\dots,P_N$, then the elementary factors of
$A^{-1}$ are $P_1^\#,\dots,P_N^\#$: this proves (i) $\Leftrightarrow$ (iii).
\item Let $B$ be a nonsingular matrix, let $C$ be a square matrix all whose elementary factors
are palindromials, and assume that $A \sim B \oplus B^{-1} \oplus C$.
Then
$$A^{-1} \sim B^{-1} \oplus B \oplus C^{-1} \sim B^{-1} \oplus B \oplus C \sim A$$
by applying (iii) $\Rightarrow$ (i) to $C$. Therefore (iv) $\Rightarrow$ (i).
\item Implication (iii) $\Rightarrow$ (iv) needs to be proved only when
$A$ is a companion matrix. Assume then that $A=C(P)$ for some palindromial $P \in \K[x]$.
Then the involution $Q \mapsto Q^\#$ must permute the irreducible factors of $P$,
therefore we may write
$$P=\prod_{i=1}^p Q_i^{\alpha_i} (Q_i^\#)^{\alpha_i} \prod_{j=1}^q R_j^{\beta_j}$$
where $Q_1,\dots,Q_p,R_1,\dots,R_q$ are irreducible and monic,
$Q_1,\dots,Q_p,Q_1^\#,\dots,Q_p^\#,R_1,\dots,R_q$ are all different,
$R_1,\dots,R_q$ are palindromials, and $\alpha_1,\dots,\alpha_p,\beta_1,\dots,\beta_q$
are positive integers.
It follows that
$$A \sim \underset{i=1}{\overset{p}{\bigoplus}} \,C(Q_i^{\alpha_i})
\,\oplus\, \underset{i=1}{\overset{p}{\bigoplus}} \,C((Q_i^\#)^{\alpha_i})
\,\oplus\, \underset{j=1}{\overset{q}{\bigoplus}} \,C(R_j^{\beta_j}).$$
Setting $B:=\underset{i=1}{\overset{p}{\bigoplus}} \,C(Q_i^{\alpha_i})$
and $C:=\underset{j=1}{\overset{q}{\bigoplus}} \,C(R_j^{\beta_j})$, we then have
$$A \sim B \oplus B^{-1} \oplus C.$$
and all the irreducible monic factors in the minimal polynomial of $C$ are palindromials.
\end{itemize}
\end{proof}

Using the same techniques as in the preceding proof,
the equivalence between statements (ii) to (v) in each one of Theorems \ref{symplectictheorem}, \ref{bilinorthotheorem}
and \ref{orthotheoremcar<>2} is obvious and we shall give no further details about it.
In each of these theorems, it remains to prove only implications
(i) $\Rightarrow$ (ii) and (v) $\Rightarrow$ (i).

\subsection{Matrices of the form $B \oplus B^{-1}$}\label{BplusB-1}

The case of matrices of the form $B \oplus B^{-1}$
is the easiest one:

\begin{prop}
Let $B \in \GL_n(\K)$. Then the matrix $B \oplus B^{-1}$
is both essentially orthogonal and essentially symplectic.
\end{prop}

\begin{proof}
A previous remark shows that $B \oplus B^{-1}$ is similar to $A:=B \oplus (B^T)^{-1}$,
so it suffices to show that $A$ is essentially orthogonal and essentially symplectic.
Setting $S:=\begin{bmatrix}
0 & I_n \\
I_n & 0
\end{bmatrix}$ and $K:=\begin{bmatrix}
0 & -I_n \\
I_n & 0
\end{bmatrix}$, we see that $S$ and $K$ are both nonsingular, with $S$ symmetric and $K$ alternate.
A straightforward computation shows that $A^TSA=S$ and $A^TKA=K$,
hence $A$ is essentially symplectic, and it is essentially orthogonal if $\car(\K) \neq 2$. \\
If nevertheless $\car(\K)=2$, set $C:=\begin{bmatrix}
0 & I_n \\
0 & 0
\end{bmatrix}$ and notice that $C+C^T=K$ is nonsingular and $A^TCA+C=0$,
hence $A$ is essentially orthogonal.
\end{proof}

Notice also that the matrix $(1) \in \Mat_1(\K)$ is essentially bilin-orthogonal (any regular
symmetric bilinear form on $\K$ is adapted to it).

\vskip 2mm
Let us now see what remains to be proved of the implication (v) $\Rightarrow$ (i)
in each theorem:
\begin{enumerate}[(a)]
\item We need to prove that for every irreducible palindromial
$P \in \K[x]$ which has no root in $\{1,-1\}$, and every integer $a \geq 1$,
the companion matrix of $P^a$ is both essentially orthogonal and essentially symplectic.
\item We need to prove that, when $\car(\K) \neq 2$, the
Jordan matrix $J_{2k+1}(1)$ is essentially orthogonal for each $k \in \N$
(in which case this is also true of $J_{2k+1}(-1)$ since it is similar to $-J_{2k+1}(1)$).
\item We need to prove that the
Jordan matrix $J_{2k}(1)$ is essentially symplectic for each $k \in \N^*$
(in which case this is also true of $J_{2k}(-1)$), and essentially orthogonal  for each $k \in \N^*$
when $\car(\K)=2$.
\end{enumerate}
Knowing this also yields Theorem \ref{orthotheoremcar2} provided Theorem \ref{symplectictheorem} holds: indeed,
it shows that if $\car(\K)=2$, then $A$ is essentially orthogonal whenever it satisfies property (v) in Theorem \ref{symplectictheorem}.

\subsection{Reducing (i) $\Rightarrow$ (ii) to the unipotent case}\label{redtounipotent}

Here, we show that the implication (i) $\Rightarrow$ (ii)
in Theorems \ref{symplectictheorem} to \ref{orthotheoremcar<>2} needs to be proved only in the case of a unipotent matrix.
We already know that a matrix that is is essentially orthogonal, essentially symplectic
or essentially bilin-orthogonal is similar to its inverse, so we will
not care anymore about this part of condition (ii).

Our starting point is the following basic lemma:

\begin{prop}\label{kerorthoim}
Let $b$ be a regular bilinear form on a finite-dimensional vector space $V$,
and assume that $b$ is symmetric or alternate.
Let $u \in \GL(V)$ be a $b$-isometry, i.e., $\forall (x,y)\in V^2, \; b(u(x),u(y))=b(x,y)$.
Let $P \in \K[x]$ be a monic polynomial with valuation $0$. Then
$\Ker P(u)=(\im P^\#(u))^{\bot_b}$.
\end{prop}

\begin{proof}
The adjoint $u^\star$ of $u$ with respect to $b$ is $u^{-1}$.
Writing $P=\underset{k=0}{\overset{n}{\sum}} a_k x^k$, with $a_n \neq 0$,
and setting $Q:=\underset{k=0}{\overset{n}{\sum}} a_{n-k} x^k$, we find that
$$P(u)^\star=P(u^\star)=P(u^{-1})=Q(u)\circ u^{-n}=a_0\, P^\#(u) \circ u^{-n},$$
therefore $\im P(u)^\star=\im P^\#(u)$ and the claimed result follows from the classical identity
$\Ker v=(\im v^\star)^{\bot_b}$, which holds for every endomorphism $v$ of $V$.
\end{proof}

\begin{cor}
Let $b$ and $u$ be as in the Proposition \ref{kerorthoim}, and let $P$ and $Q$
be monic polynomials with valuation $0$ such that $P^\#$ is prime with $Q$.
Then $\Ker P(u) \bot_b \Ker Q(u)$.
\end{cor}

\begin{proof}
Indeed, $\Ker Q(u) \subset \im P^\#(u)$ since $P^\#$ is prime with $Q$.
\end{proof}

With the same assumptions, assume further that $\car(\K) \neq 2$ and split the minimal polynomial
$\mu$ of $u$ as $\mu=R(x)\,(x-1)^p(x+1)^q$ where $R$ has no root in $\{1,-1\}$, hence neither does
$R^\#$. The previous corollary and the kernel decomposition theorem show that
$V=\Ker R(u) \overset{\bot_b}{\oplus} \Ker (u-\id)^p \overset{\bot_b}{\oplus} \Ker(u+\id)^q$,
hence $\Ker (u-\id)^p$ and $\Ker (u+\id)^q$ are both regular $b$-spaces:
we deduce that the restrictions of $u$ to $\Ker (u-\id)^p$ and $\Ker (u+\id)^q$
are both isometries for regular bilinear forms which are
symplectic (respectively, symmetric) if $b$ is symplectic (respectively, symmetric): moreover,
if $u$ is essentially orthogonal, then
its restrictions to $\Ker (u-\id)^p$ and $\Ker (u+\id)^q$ are essentially orthogonal.
This leaves us with only two cases: $u-\id$ is nilpotent or $u+\id$ is nilpotent.
However, in the second case, $(-u)-\id$ is nilpotent hence only the first case
needs to be addressed (see Section \ref{basicred}).

The case $\car(\K)=2$ is handled similarly and even more easily since
only the eigenvalue $1$ needs to be taken into account.

In order to prove implication (i) $\Rightarrow$ (ii)
in Theorems \ref{symplectictheorem} to \ref{orthotheoremcar<>2}, it suffices to prove the following result:

\begin{prop}\label{decortiquage}
Let $b$ be a regular bilinear form on a finite-dimensional vector space $V$, and
$u \in \GL(V)$ be a $b$-isometry, i.e., $\forall (x,y)\in V^2, \; b(u(x),u(y))=b(x,y)$.
Assume that $u-\id_V$ is nilpotent.
\begin{enumerate}[(a)]
\item If $b$ is symplectic, then, for every
$k \in \N$, the number of Jordan blocks of $u$ with size $2k+1$ is even.
\item If $\car(\K)=2$ and $b$ is symmetric, then, for every
$k \in \N^*$, the number of Jordan blocks of $u$ with size $2k+1$ is even.
\item If $\car(\K) \neq 2$ and $b$ is symmetric, then, for every
$k \in \N^*$, the number of Jordan blocks of $u$ with size $2k$ is even.
\end{enumerate}
\end{prop}

\section{The case of unipotent matrices}

We start by giving two proofs of Proposition \ref{decortiquage}; the first one is short. The second one is substantially longer and may
thus be skipped at first reading; it is, however, unavoidable in order to grasp fully the discussion featured in Section \ref{car2quadSec}.
Let $b$ be a regular bilinear form, symmetric or alternate, on a finite-dimensional vector space $V$, and
$u \in \GL(V)$ be a $b$-isometry such that $u-\id$ is nilpotent.

\subsection{Short proof of Proposition \ref{decortiquage}}

\begin{enumerate}[(i)]
\item Assume that $b$ is symplectic. It then suffices to prove that $\rk (u-\id)^{2k}$ is even for every $k \in \N$. \\
For every $k \in \N$ and every $x \in V$, one has
\begin{align*}
b(u^k(x),(u-\id)^{2k}(x)) & =b\bigl((u^{-1}-\id)^k(u^k(x)),(u-\id)^k(x)\bigr) \\
& = (-1)^k\, b((u-\id)^k(x),(u-\id)^k(x))=0.
\end{align*}
This shows that the bilinear form $(x,y) \mapsto b(u^k(x),(u-\id)^{2k}(y))$ is alternate, hence
its rank is even. However its rank is that of $(u-\id)^{2k}$
since $(x,y) \mapsto b(u^k(x),y)$ is non-degenerate (because $u \in \GL(V)$).

\item Assume now that $b$ is symmetric and $\car(\K)=2$. It suffices to prove that $\rk (u-\id)^{2k}$
is even for every $k \in \N^*$. \\
Let $k \in \N^*$ and $x \in V$, and set $y:=(u-\id)^{k-1}(x)$. With the same computation as in (i),
$$b(u^k(x),(u-\id)^{2k}(x))=b((u-\id)(y),(u-\id)(y))=b(u(y),u(y))+b(y,y)=0.$$
As in (i), this shows that $\rk(u-\id)^{2k}$ is even.

\item Assume that $b$ is symmetric and $\car(\K)\neq 2$. It then suffices to prove that $\rk (u-\id)^{2k+1}$ is even for every $k \in \N$. \\
For every $k \in \N$ and every $x \in V$, setting $y:=(u-\id)^k(x)$, one finds:
\begin{align*}
b(u^k(u+\id)(x),(u-\id)^{2k+1}(x))
& =b((\id-u)^k(u+\id)(x),(u-\id)^{k+1}(x)) \\
& =(-1)^k b((u+\id)(y),(u-\id)(y)) \\
& =(-1)^k \bigl(b(u(y),u(y))-b(y,y)\bigr)=0.
\end{align*}
Setting $c : (x,y) \mapsto b((u^k(u+\id)(x),y)$, we deduce that
$(x,y) \mapsto c(x,(u-\id)^{2k+1}(y))$ is alternate, hence its rank is even.
However this rank is that of $(u-\id)^{2k+1}$ since $c$ is non-degenerate (indeed $b$ is non-degenerate and $u^k\circ (u+\id)$
is an automorphism of $V$ since $\car(\K) \neq 2$ and $u$ is unipotent).
\end{enumerate}

\subsection{Long proof of Proposition \ref{decortiquage}}\label{decorsection}

Set $n:=\dim V$.
In this proof, orthogonality is always considered with respect to $b$ unless specified otherwise. \\
Using the Jordan reduction theorem ``block-wise" yields a decomposition\footnote{For convenience purpose, we use $2n$ instead of $n$ as an upper bound
for the first direct sum.}:
$$V=\underset{k=1}{\overset{2n}{\oplus}} \underset{i=1}{\overset{k}{\oplus}} V_{k,i}$$
where, for every $k \in \lcro 1,n\rcro$, some $V_{k,i}$ might be $\{0\}$ and:
\begin{itemize}
\item for every $i \in \lcro 1,k-1\rcro$, the linear map $u-\id$ induces an isomorphism from $V_{k,i}$ to $V_{k,i+1}$ ;
\item $(u-\id)(x)=0$ for every $x \in V_{k,k}$.
\end{itemize}

Then $\Ker (u-\id)^{k-1} \oplus \underset{i=k}{\overset{2n}{\oplus}} V_{i,i-k+1}=\Ker (u-\id)^k$ for every $k \in \N^*$. For each
$k \in \lcro 1,n\rcro$, set
$$F_k:=V_{2k-1,k} \quad \; \quad G_k:=V_{2k,k} \quad \text{and} \quad H_k=V_{2k,k+1.}$$
The following diagram accounts for the action of $u-\id$ on the various spaces we have just defined.
$$\boxed{F_1} \rightarrow \{0\}$$
$$\boxed{G_1} \overset{\sim}{\rightarrow} \boxed{H_1} \rightarrow \{0\}$$
$$V_{3,1} \overset{\sim}{\rightarrow} \boxed{F_2} \overset{\sim}{\rightarrow} V_{3,3} \rightarrow \{0\}$$
$$\ldots\ldots\ldots\ldots\ldots\ldots\ldots\ldots\ldots\ldots\ldots\ldots\ldots\ldots\ldots\ldots\ldots$$
$$\hskip 2mm V_{2k,1} \overset{\sim}{\rightarrow} \cdots  \overset{\sim}{\rightarrow}V_{2k,k-1}
\overset{\sim}{\rightarrow} \boxed{G_k} \overset{\sim}{\rightarrow} \boxed{H_k}
\overset{\sim}{\rightarrow} V_{2k,k+2}
\overset{\sim}{\rightarrow} \cdots \overset{\sim}{\rightarrow} V_{2k,2k} \rightarrow \{0\} $$
$$\hskip 8mm V_{2k+1,1} \overset{\sim}{\rightarrow} \cdots  \overset{\sim}{\rightarrow}V_{2k+1,k}
\overset{\sim}{\rightarrow} \boxed{F_{k+1}}
\overset{\sim}{\rightarrow} V_{2k+1,k+2}
\overset{\sim}{\rightarrow} \cdots \overset{\sim}{\rightarrow} V_{2k+1,2k+1} \rightarrow \{0\}$$
$$\ldots\ldots\ldots\ldots\ldots\ldots\ldots\ldots\ldots\ldots\ldots\ldots\ldots\ldots\ldots\ldots\ldots\ldots\ldots
\ldots\ldots\ldots\ldots\ldots\ldots\ldots$$
Set
$$F:=\underset{k=1}{\overset{n}{\oplus}} F_k, \quad  G:=\underset{k=1}{\overset{n}{\oplus}} G_k \quad \text{and} \quad
H:=\underset{k=1}{\overset{n}{\oplus}} H_k,$$
and
$$E:=\underset{k=1}{\overset{2n}{\oplus}}  \underset{i=[(k+1)/2]+1}{\overset{k}{\oplus}} V_{k,i}.$$
Notice that $\dim F_k$ (respectively, $\dim G_k=\dim H_k$)
is the number of Jordan blocks of size $2k-1$ (respectively, $2k$) for $u$. \\
Proposition \ref{kerorthoim} yields:
$$\forall k \in \N^*, \; \Ker(u-\id)^k = \Bigl[\im (u-\id)^k\Bigr]^\bot.$$
For every $k \in \lcro 1,n\rcro$ and $i \in \lcro 1,k\rcro$,
we see that $V_{k,i} \subset \Ker (u-\id)^{k+1-i}$ and $V_{k,i} \subset \im (u-\id)^{i-1}$, and it follows that
$$E \,\bot\,\bigl(E \oplus F \oplus G \oplus H).$$
On the other hand, we notice that $\codim_V E=\dim(E \oplus F \oplus G \oplus H)$, therefore
$$E^\bot=E \oplus F \oplus G \oplus H$$
and we deduce that $F \oplus G \oplus H$ is $b$-regular. \\
Using again the relation $\Ker (u-\id)^k \bot \im (u-\id)^k$ for each $k \in \N^*$, we find that
for every $(k,l)\in \lcro 1,n\rcro^2$, $k\leq l$ implies $H_k \bot H_l$, and $k<l$ implies $F_k \bot F_l$ and $H_k \bot G_l$. \\
Finally, we work with
 $W:=F\oplus G \oplus H$ equipped with the (symmetric or alternate) regular bilinear form
 $b_W$ induced by $b$, and we consider the endomorphism $v$ of $W$ that coincides with $u$ on $G$
 and is the identity on $F \oplus H$:
 since $(u-\id)(F \oplus H)$ is included in $E$, and is therefore orthogonal to $W$, we find that
 $v$ is a $b_W$-isometry with $\Ker(v-\id)=F\oplus H$ and $\im (v-\id)=H$.

\begin{claim}
For each $k \in \lcro 1,n\rcro$, the subspaces $F_k$ and $G_k \oplus H_k$ are $b$-regular.
\end{claim}

\begin{proof}
Notice that $H^\bot=\im (v-\id)^\bot=\Ker (v-\id)=F \oplus H$ (orthogonality is now considered with respect to $b_W$). \\
We deduce that both $F$ and $G \oplus H$ are $b$-regular.
Since $F=F_1 \overset{\bot}{\oplus} F_2 \overset{\bot}{\oplus} \cdots \overset{\bot}{\oplus} F_n$,
it follows that $F_1,\dots,F_n$ are all $b$-regular. \\
Notice then that $H_1$ is orthogonal to $H_1 \oplus \underset{k=2}{\overset{n}{\oplus}} (G_k \oplus H_k)$.
Since $\dim G_1=\dim H_1$, we deduce that the orthogonal subspace of $H_1$ in $H \oplus G$ is $\underset{k=2}{\overset{n}{\oplus}} (G_k \oplus H_k)$,
which shows that both $G_1 \oplus H_1$ and $\underset{k=2}{\overset{n}{\oplus}} (G_k \oplus H_k)$ are $b$-regular.
Continuing by induction, we find that $G_k \oplus H_k$ is $b$-regular for each $k \in \lcro 1,n\rcro$.
\end{proof}

At this point, we may prove Proposition \ref{decortiquage} by distinguishing between three cases.

\begin{enumerate}[(a)]
\item Assume that $b$ is symplectic. Then, for each $k \in \lcro 1,n\rcro$,
the restriction of $b$ to $F_k \times F_k$ is symplectic, which shows that $\dim F_k$ is even.
\item Assume that $\car(\K)=2$ and $b$ is symmetric. Let $k \in \lcro 2,n\rcro$.
The key point is that the quadratic form $x \mapsto b(x,x)$ vanishes on $F_k$.
Indeed, given $x \in F_k$, we may find some $y \in V_{2k-1,k-1}$ such that $x=u(y)-y$, hence
$b(x,x)=b(u(y),u(y))+b(y,y)=0$ since $b$ is skew-symmetric. It follows that $b$ induces a symplectic
form on $F_k$, hence $\dim F_k$ is even.
\item Assume finally that $\car(\K) \neq 2$ and $b$ is symmetric. Let $k \in \lcro 1,n\rcro$
and denote by $v_k$ the endomorphism of $G_k \oplus H_k$ induced by $v$. Set $p:=\dim H_k$.
Then $H_k$ is a totally isotropic subspace for $b$, hence we may find an hyperbolic basis $\bfB$
of $G_k \oplus H_k$ whose first $p$ vectors belong to $H_k$.
Since $H_k=\Ker(v_k-\id)=\im(v_k-\id)$, we find that $\Mat_\bfB(v_k)=\begin{bmatrix}
I_p & A \\
0 & I_p
\end{bmatrix}$ for some $A \in \GL_p(\K)$. Since $\bfB$ is hyperbolic and $v_k$ is $b_{G_k \oplus H_k}$-orthogonal,
a straightforward computation shows that $A$ is skew-symmetric. Since $\car(\K) \neq 2$, this shows that $\dim H_k=p$ is even.
\end{enumerate}
The proof of Proposition \ref{decortiquage} is now complete, and it follows that implication (i) $\Rightarrow$ (ii)
in Theorems \ref{symplectictheorem}, \ref{bilinorthotheorem} and \ref{orthotheoremcar<>2} is proved.

\subsection{Jordan blocks with eigenvalue $1$}\label{Jordanblock}

Our aim here is to prove the following two results, which are already known
in the case $\car(\K) \neq 2$. We reproduce the proof for the sake of completeness
and because the strategy is to be reused later on.

\begin{prop}\label{jordaneven}
Let $n \in \N^*$. Then the Jordan matrix $J_{2n}(1)$ is essentially symplectic. If $\car(\K)=2$, then $J_{2n}(1)$ is also essentially orthogonal.
\end{prop}

\begin{prop}\label{jordanodd}
Assume that $\car(\K) \neq 2$.
Let $n \in \N$. Then the Jordan matrix $J_{2n+1}(1)$ is essentially orthogonal.
\end{prop}

\begin{proof}[Proof of Proposition \ref{jordaneven}]
Let $A=(a_{i,j}) \in \Mat_{2n}(\K)$. A straightforward computation shows that $J_{2n}(1)^TAJ_{2n}(1)=A$ if and only if both of the following conditions are satisfied:
\begin{enumerate}[(i)]
\item $a_{i,j-1}+a_{i-1,j}+a_{i-1,j-1}=0$ for every $(i,j)\in \lcro 2,2n\rcro^2$;
\item $a_{k,1}=a_{1,k}=0$ for every $k\in \lcro 1,2n-1\rcro$.
\end{enumerate}
We construct such a matrix $A \in \Mat_{2n}(\K)$ as follows:
\begin{itemize}
\item we set $a_{i,j}:=0$ whenever $i+j<2n+1$;
\item we set $a_{i,2n+1-i}:=(-1)^i$ for every $i \in \lcro 1,2n\rcro$;
\item we set $a_{i,j}:=0$ whenever $i>n$ and $j>n$;
\item we then define (doubly)-inductively $a_{i,j}$ for $i$ from $n$ down to $2$ and for
$j$ from $2n-i+2$ up to $2n$ by $a_{i,j}:=-a_{i,j-1}-a_{i+1,j-1}$;
\item symmetrically, we define $a_{i,j}$ for $j$ from $n$ down to $2$ and for
$i$ from $2n-j+2$ up to $2n$ by $a_{i,j}:=-a_{i-1,j}-a_{i-1,j+1}$.
\end{itemize}
One checks that $A$ satisfies conditions (i) and (ii). Moreover, $A$
has the form
$$\begin{bmatrix}
0 & 0 & \dots & 0 & 1 \\
0 & &  & -1 & * \\
\vdots & & \iddots & & \\
0 & 1 &  & * & * \\
-1 & * &  & * & *
\end{bmatrix}$$
and hence $A$ is nonsingular.
Had we replaced the second point by $a_{2n+1-i,i}:=(-1)^{i+1}$, the matrix would have been
$A^T$ since the other conditions are symmetrical. Since $a_{2n+1-i,i}=(-1)^{2n-i}=-(-1)^{i+1}$
and all the other conditions are linear, this shows that $A^T=-A$. As all the diagonal entries of $A$ have been set to zero,
this shows that $A$ is alternate. Therefore $J_{2n}(1)$ is essentially symplectic. \\
Assume finally that $\car(\K)=2$ and choose an arbitrary symplectic form $b$ for which $u : X \mapsto J_{2n}(1)X$ is a symplectic morphism.
Then a (regular) quadratic form $q$ with polar form $b$ is determined by choosing $(q(e_1),\dots,q(e_{2n}))$ arbitrarily
in $\K^{2n}$ (where $(e_1,\dots,e_{2n})$ is the canonical basis of $\K^{2n}$). Given such a form $q$,
the map $u$ is $q$-orthogonal if and only if $q(u(e_i))=q(e_i)$ for every $i \in \lcro 1,2n\rcro$,
which is equivalent to having $q(e_{i-1})=-b(e_i,e_{i-1})$ whenever $i \geq 2$.
Obviously one may find a quadratic form $q$ which fits these conditions (and we may even choose $q(e_{2n})$ as we please).
\end{proof}

\begin{proof}[Proof of Proposition \ref{jordanodd}]
Using the same arguments as in the previous proof, we see that it suffices to find a nonsingular symmetric matrix $A=(a_{i,j}) \in \Mat_{2n+1}(\K)$
which satisfies:
\begin{enumerate}[(i)]
\item $a_{i,j-1}+a_{i-1,j}+a_{i-1,j-1}=0$ for every $(i,j)\in \lcro 2,2n+1\rcro^2$;
\item $a_{k,1}=a_{1,k}=0$ for every $k\in \lcro 1,2n\rcro$.
\end{enumerate}
We construct such a matrix $A \in \Mat_{2n+1}(\K)$ as follows:
\begin{itemize}
\item we set $a_{i,j}:=0$ whenever $i+j<2n+2$;
\item we set $a_{i,2n+2-i}:=(-1)^i$ for every $i \in \lcro 1,2n+1\rcro$;
\item we set $a_{i,j}:=0$ whenever $i>n+1$ and $j>n+1$;
\item we set $a_{i,n+1}:=\dfrac{(-1)^i}{2}$ whenever $i>n+1$ and $a_{n+1,j}:=\dfrac{(-1)^j}{2}$ whenever $j>n+1$;
\item we then define (doubly)-inductively $a_{i,j}$ for $i$ from $n$ down to $2$ and for
$j$ from $2n-i+3$ up to $2n+1$ by $a_{i,j}:=-a_{i,j-1}-a_{i+1,j-1}$;
\item symmetrically, we define $a_{i,j}$ for $j$ from $n$ down to $2$ and for
$i$ from $2n-j+3$ up to $2n+1$ by $a_{i,j}:=-a_{i-1,j}-a_{i-1,j+1}$.
\end{itemize}
One checks that $A$ satisfies conditions (i) and (ii) and is both symmetric and nonsingular,
the key points being that $a_{n+1,n+1}+a_{n+1,n+2}+a_{n+2,n+1}=0$ and $a_{n+1,n+2}=a_{n+2,n+1}$ (this is precisely where we need
the assumption on the characteristic of $\K$).
\end{proof}

\section{The case of elementary companion matrices}\label{companionsection}

In order to conclude our proof of Theorems \ref{symplectictheorem}, \ref{bilinorthotheorem}, and \ref{orthotheoremcar<>2},
we must prove implication (v) $\Rightarrow$ (i) in each of them.
By the considerations of Sections \ref{Asyminverse}, \ref{BplusB-1},
and \ref{Jordanblock}, we must prove the following proposition:

\begin{prop}\label{companionprop}
Let $P \in \K[x]$ be an irreducible palindromial with no root in $\{1,-1\}$, and let $a \in \N^*$.
Then the companion matrix $C(P^a)$ is both essentially orthogonal and essentially symplectic.
\end{prop}

We distinguish between two cases, whether $\K$ is finite or not.

\subsection{The case of an infinite field}

Assume that $\K$ is infinite, and notice that Proposition \ref{companionprop} holds trivially in $\overline{\K}$:
indeed, any palindromial $Q$ of $\overline{\K}[x]$ with degree $1$ is $x-\lambda$ for some
$\lambda \in \overline{\K} \setminus \{0\}$ such that $\frac{1}{\lambda}=\lambda$, hence
$\lambda=\pm 1$.
It follows that Theorems \ref{symplectictheorem}, \ref{bilinorthotheorem} and \ref{orthotheoremcar<>2} all hold for the field $\overline{\K}$.
It thus suffices to prove the following result:

\begin{prop}\label{infiniteextension}
Assume that $\K$ is infinite and let $\L$ be a field extension of $\K$.
Let $A \in \Mat_n(\K)$, and assume that $A$ is essentially orthogonal (respectively, essentially symplectic, respectively, essentially bilin-orthogonal)
in $\Mat_n(\L)$. Then $A$ is essentially orthogonal (respectively, essentially symplectic, respectively, essentially bilin-orthogonal)
in $\Mat_n(\K)$.
\end{prop}

\begin{proof}
The line of reasoning here is classical.
\begin{itemize}
\item Assume that $A^TSA=S$ for some nonsingular symmetric matrix $S \in \Mat_n(\L)$. Choose a basis
$(b_1,\dots,b_p)$ of the $\K$-linear subspace of $\L$ spanned by the entries of $S$, and
split up $S=b_1S_1+\cdots+b_pS_p$ where $S_1,\dots,S_p$ are symmetric matrices of $\Mat_n(\K)$.
Since $A \in \Mat_n(\K)$, we find that $A^T S_k A=S_k$ for every $k \in \lcro 1,p\rcro$,
hence $A^T(x_1S_1+\cdots+x_pS_p)A=x_1S_1+\cdots+x_pS_p$ for every $(x_1,\dots,x_p)\in \K^p$. \\
The polynomial $\det(x_1S_1+\cdots+x_pS_p) \in \K[x_1,\dots,x_p]$ is nonzero since
$\det(b_1S_1+\cdots+b_pS_p) \neq 0$. As $\K$ is infinite, we deduce that
we may find some
$(x_1,\dots,x_p) \in \K^p$ such that $\det(x_1S_1+\cdots+x_pS_p) \neq 0$.
It follows that $S':=x_1S_1+\cdots+x_pS_p$ is a nonsingular symmetric matrix of $\Mat_n(\K)$ which satisfies
$A^TS'A=S'$. \\
This shows that $A$ is essentially bilin-orthogonal over $\K$ if it is essentially bilin-orthogonal over $\L$.
\item A similar argument shows that $A$ is
essentially symplectic over $\K$ if it is essentially symplectic over $\L$.
\item Assume finally that $\car(\K) \neq 2$ and $A$ is essentially orthogonal over $\L$.
Choose $M \in \Mat_n(\L)$ such that $A^TMA+M$ is alternate and $M+M^T$ is nonsingular.
As before, split up $M=b_1M_1+\cdots+b_pM_p$ where $M_1,\dots,M_p$ are all matrices of $\Mat_n(\K)$
with $A^TM_kA+M_k$ alternate for each $k \in \lcro 1,p\rcro$, and $(b_1,\dots,b_p)\in \L^p$ is linearly independent over $\K$.
With the same argument as before, we see that we may find
$(x_1,\dots,x_p) \in \K^p$ such that $x_1(M_1+M_1^T)+\cdots+x_p(M_p+M_p^T)$ is nonsingular, in which case
$M':=x_1M_1+\cdots+x_pM_p$ is such that $A^TM'A+M'$ is alternate and $M'+(M')^T$ is nonsingular.
This shows that if $A$ is essentially orthogonal over $\L$, then it is essentially orthogonal over $\K$.
\end{itemize}
\end{proof}

As a consequence of Theorems \ref{symplectictheorem} to \ref{orthotheoremcar2},
Proposition \ref{infiniteextension} still holds when $\K$ is finite, although we do not know any direct proof of it.

\subsection{The case of a finite field}\label{finitecompanion}

When $\K$ is finite, we know that $P$, being irreducible, must be separable (i.e.,
it has distinct roots $x_1,\dots,x_n$ in $\overline{\K}$). As we have seen earlier, we must have $n \geq 2$ (in fact, $n$ is even
since no root of $P$ in $\overline{\K}$ is fixed by the involution $a \mapsto a^{-1}$).
We introduce the quotient field $\L:=\K[x]/(P(x))$ and denote by $y$ the class of $x$ in it.
Since $P$ is a palindromial, $y^{-1}$ is another root of $P$ in $\L$
hence we may find an automorphism $\sigma$ of $\L$ over $\K$ such that
$\sigma(y)=y^{-1}$. It follows that $\sigma^2(y)= \sigma(y^{-1})=\sigma(y)^{-1}=y$,
hence $\sigma^2=\id$ because $\L=\K[y]$. \\
We then define the subfield $\K'=\{z \in \L : \; \sigma(z)=z\}$ of $\L$
and notice that $\L$ is a quadratic extension of $\K'$:
indeed, $\L=\K'[y]$, and $(X-y)(X-\sigma(y))$ is the minimal polynomial of $y$ on $\K'$
since $y \not\in \K'$, $y\sigma(y)=1$ and $y+\sigma(y) \in \K'$ (because $\sigma^2=\id$). \\
By an hermitian matrix of $\Mat_k(\L)$, we mean an hermitian matrix with respect to the quadratic extension
$\K'-\L$, i.e., a matrix $H \in \Mat_k(\L)$ which satisfies $\sigma(H)^T=H$.

Let us come back to the matrix $C(P^a)$ and remark that
$$C(P^a) \sim J_a(1) \otimes C(P).$$
There are numerous ways to prove this: we note that, over $\overline{\K}$,
\begin{align*}
C(P^a) & \sim C((x-x_1)^a) \oplus C((x-x_2)^a) \oplus \cdots \oplus C((x-x_n)^a) \\
& \sim J_a(x_1) \oplus J_a(x_2) \oplus \cdots \oplus J_a(x_n) \\
& \sim x_1\,J_a(1) \oplus \cdots \oplus x_n\,J_a(1) \\
& \sim J_a(1) \otimes \Diag(x_1,\dots,x_n)  \\
& \sim J_a(1) \otimes C(P),
\end{align*}
and invoke the invariance of similarity when the ground field is extended.
It then suffices to prove that $A:=J_a(1) \otimes C(P)$ is both essentially orthogonal and essentially symplectic.
Now set $D:=I_a \otimes C(P)$
the block-diagonal matrix with $a$ diagonal blocks all equal to $C(P)$.
Then $D$ has $P$ as minimal polynomial hence $D$ induces a structure of $\L$-vector space on $V:=\K^{na}$.
Moreover $u : X \mapsto AX$ is $\L$-linear since $A$ commutes with $D$,
and $u$ is represented by the matrix $y.J_a(1)$ in some basis of the $\L$-vector space $V$. \\
We need the following result, which we will prove later:

\begin{lemme}\label{existhermitian}
There is a nonsingular hermitian matrix $H \in \Mat_a(\L)$ such that $J_a(1)$ is $H$-unitary, i.e.,
$\sigma(J_a(1))^THJ_a(1)=H$, i.e., $J_a(1)^THJ_a(1)=H$.
\end{lemme}

Fix such a matrix $H$ and denote by $b : (X,Y) \mapsto \sigma(X)^THY$ the hermitian product on $V$
(identified with $\L^a$) associated to it. Since $y\sigma(y)=1$, we find that $y.\id_V$ is $b$-unitary,
and we then deduce from the assumptions that $u$ is $b$-unitary.
Set finally
$$q : X \mapsto \Tr_{\K'/\K}\bigl(b(X,X)\bigr)$$
and notice that $q$ is a quadratic form on the $\K$-vector space $V$ for which $u$ is orthogonal. \\
Notice that for every $(X,Y)\in V^2$,
$$q(X+Y)-q(X)-q(Y)=\Tr_{\K'/\K}\bigl(b(X,Y)+b(Y,X))\bigr)
=\Tr_{\K'/\K}\bigl(\Tr_{\L/\K'}(b(X,Y))\bigr)=\Tr_{\L/\K} b(X,Y),$$
and since $b$ is non-degenerate, it follows that the polar form of $q$ is also non-degenerate (whatever the value of $\car(\K)$). \\
Assume finally that $\car(\K) \neq 2$. We may then choose $\varepsilon \in \L \setminus \K'$ such that
$\sigma(\varepsilon)=-\varepsilon$, and classically
$$(X,Y) \mapsto \varepsilon\,(b(X,Y)-b(Y,X))$$
is a symplectic form on the $\K'$-vector space $V$ for which $u$ is a symplectic morphism.
It follows that
$$(X,Y) \mapsto \Tr_{\K'/\K}\bigl(\varepsilon(b(X,Y)-b(Y,X))\bigr)$$
is a symplectic form on the $\K$-vector space $V$ for which $u$ is a symplectic morphism.
We may then finish the proof of Proposition \ref{companionprop} by establishing Lemma \ref{existhermitian}.

\begin{proof}[Proof of Lemma \ref{existhermitian}]
Assume first that $\car(\K) \neq 2$, and choose $\varepsilon \in \L \setminus \{0\}$ such that $\sigma(\varepsilon)=-\varepsilon$.
\begin{itemize}
\item If $a$ is odd, we use Proposition \ref{jordanodd} to find a nonsingular symmetric matrix $S \in \Mat_a(\K)$ such that $J_a(1)^TS J_a(1)=S$
and we set $H:=S$.
\item If $a$ is even, we use Proposition \ref{jordaneven} to find a nonsingular alternate matrix $N \in \Mat_a(\K)$ such that $J_a(1)^TN J_a(1)=N$, and we set $H:=\varepsilon N$.
\end{itemize}
In any case, $H$ has the claimed properties. \\
Assume now that $\car(\K)=2$. If $a$ is even, we find a nonsingular alternate matrix $N \in \Mat_n(\K)$ such that $J_a(1)^TN J_a(1)=N$, and we notice
that $H:=N$ is hermitian. Assume finally that $a$ is odd. Then $\K'=\Ker(\sigma+\id)$. Choose $\alpha \in \L \setminus \K'$
and notice that $\beta:=\alpha+\sigma(\alpha)$ belongs to $\K'\setminus\{0\}$. We write $a=2b+1$ for some integer $b$, and define
$H \in \Mat_a(\K)$ as follows:
\begin{itemize}
\item we set $h_{i,j}:=0$ whenever $i+j<2b+2$;
\item we set $h_{i,2b+2-i}:=\beta$ for every $i \in \lcro 1,2b+1\rcro$;
\item we set $h_{i,j}:=0$ whenever $i>b+1$ and $j>b+1$;
\item we set $h_{i,b+1}:=\alpha$ whenever $i>b+1$ and $h_{b+1,j}:=\sigma(\alpha)$ whenever $j>b+1$;
\item we then define (doubly)-inductively $h_{i,j}$ for $i$ from $b$ down to $2$ and for
$j$ from $2b-i+3$ up to $2b+1$ by $h_{i,j}:=-h_{i,j-1}-h_{i+1,j-1}$;
\item symmetrically, we define $h_{i,j}$ for $j$ from $b$ down to $2$ and for
$i$ from $2b-j+3$ up to $2b+1$ by $h_{i,j}:=-h_{i-1,j}-h_{i-1,j+1}$.
\end{itemize}
As in the proof of Proposition \ref{jordaneven}, one shows that $H$ is nonsingular, hermitian,
and $J_a(1)^THJ_a(1)=H$.
\end{proof}

This finishes our proof of implication (v) $\Rightarrow$ (i) in Theorems \ref{symplectictheorem}, \ref{bilinorthotheorem}
and \ref{orthotheoremcar<>2}. Therefore, all those theorems are proved, and
Theorem \ref{orthotheoremcar2} follows from them and from Proposition \ref{companionprop}, as explained earlier.

\section{Refinements for symmetric bilinear forms in characteristic 2}\label{car2bilinSec}

In this section, we assume that $\K$ has characteristic $2$ and is \emph{perfect} (e.g.\ $\K$ is finite or algebraically closed).
Let $n \in \N^*$. Then the following results hold (see Chapter XXXV of \cite{invitquad}):
\begin{itemize}
\item if $n$ is odd, the matrix $I_n$ is, up to congruence, the sole
nonsingular symmetric matrix of $\Mat_n(\K)$;
\item for every $n \in \N$, the matrix $I_n$ is, up to congruence, the sole
nonsingular nonalternate symmetric matrix of $\Mat_n(\K)$.
\end{itemize}

When $n$ is odd, we have successfully determined the Jordan canonical forms of the elements in the group $\text{O}_n(\K)$.
Assume, for the rest of the section, that $n$ is even. Then we have two congruence classes of symmetric matrices in $\Mat_n(\K)$:
the one of $I_n$ and the one of $\begin{bmatrix}
0 & I_{n/2} \\
I_{n/2} & 0
\end{bmatrix}$. We have already classified the essentially symplectic morphisms, so
we are now interested in the automorphisms that are orthogonal for some regular nonalternate symmetric bilinear form.
A necessary condition for having this property is the following:

\begin{prop}
Let $u \in \GL(V)$ and assume that there is a regular nonalternate symmetric bilinear form $b$ for which $u$ is orthogonal. \\
Then $1$ is an eigenvalue of $u$.
\end{prop}

\begin{proof}
We lose no generality in assuming that $V=\K^n$ and $b : (X,Y) \mapsto X^TY$. Set $q : X \mapsto b(X,X)$,
and notice that $q(x_1,\dots,x_n)=(x_1+\dots+x_n)^2$ for every $(x_1,\dots,x_n) \in \K^n$.
Since $q \circ u=q$ and $\car(\K)=2$, we deduce that the linear form $f : (x_1,\dots,x_n) \mapsto x_1+\dots+x_n$ satisfies $f \circ u=f$.
This proves that $1$ is an eigenvalue of the transposed endomorphism $u^T : (\K^n)^\star \rightarrow (\K^n)^\star$,
hence $1$ is an eigenvalue of $u$.
\end{proof}

\begin{Rem}\label{soleremark}
The result actually holds for an arbitrary field of characteristic $2$, with a subtler proof
(see exercise 17 in chapter XXXV of \cite{invitquad}).
\end{Rem}

We shall see that the converse is true, which leads to the following theorem:

\begin{theo}\label{perfectcar2}
Assume that $\K$ is perfect of characteristic $2$, and let $n \in \N^*$.
Let $A \in \GL_{2n}(\K)$.
The following conditions are equivalent:
\begin{enumerate}[(i)]
\item $A$ is similar to a matrix of $\Ortho_{2n}(\K)$;
\item $A$ is essentially symplectic and $1$ is an eigenvalue of $A$.
\end{enumerate}
\end{theo}

For an arbitrary field of characteristic $2$, condition (ii) characterizes the even-sized matrices that are orthogonal
for at least one regular nonalternate symmetric bilinear form (use Remark \ref{soleremark}).

\vskip 3mm
In order to prove the theorem, we consider an essentially bilin-orthogonal morphism $u \in \GL(V)$,
with $\dim V$ even, and assume that $1$ is an eigenvalue of $u$. We need to find a nonalternate regular symmetric bilinear form
$b$ on $V$ for which $u \in O(b)$. Notice that since $\dim V$ is even, three situations may arise:
\begin{enumerate}[(a)]
\item There is a Jordan block of size $1$ for the eigenvalue $1$ of $u$, in which case
the proof of implication (v) $\Rightarrow$ (i) in Theorem \ref{bilinorthotheorem} gives an explicit
construction of a nonalternate $b$ (the key point being that $\id_\K$ is orthogonal for the nonalternate form $(x,y) \mapsto xy$).
\item At least one Jordan block of $u$ for the eigenvalue $1$ is even-sized: in this case,
we define $A$ almost as in the proof of Proposition \ref{jordaneven}
but set $a_{2n,2n}=1$ instead of $a_{2n,2n}=0$. One checks that this yields
a regular nonalternate symmetric bilinear form for which $X \mapsto J_{2n}X$ is orthogonal.
Using the method of the proof of implication (v) $\Rightarrow$ (i) in Theorem \ref{bilinorthotheorem},
we find a well-suited $b$ for $u$.
\item There is an integer $k \geq 1$ such that $u$ has a Jordan block of size $2k+1$ for the eigenvalue $1$;
since $u$ is essentially bilin-orthogonal, it has an even number of such blocks (with $k$ fixed); rather than
use all these blocks to form a matrix of type $B \oplus B^{-1}$, we may then keep
a pair of them separated from the rest and try to prove directly that their direct sum is orthogonal for some
nonalternate regular symmetric bilinear form. If this is true, then the same arguments as before show
that we may find a well-suited $b$ for $u$.
\end{enumerate}

\noindent In order to conclude our proof of Theorem \ref{perfectcar2}, it suffices to establish the following lemma:

\begin{lemme}\label{pairoddjordan1}
Let $n \in \N^*$. Then there is a nonalternate nonsingular symmetric matrix $S \in \Mat_{4n+2}(\K)$
such that $J_{2n+1}(1) \oplus J_{2n+1}(1)$ is $S$-orthogonal.
\end{lemme}

\begin{proof}
We work with
$$J_{2n+1}(1) \otimes I_2=\begin{bmatrix}
I_2 & I_2 & & (0)  \\
0 & I_2 & \ddots & \\
 & & \ddots & I_2\\
(0) & & 0 & I_2
\end{bmatrix},$$
which is similar to $J_{2n+1}(1) \oplus J_{2n+1}(1)$.
We search for a suitable $S$ of the form
$$\begin{bmatrix}
S_{1,1} & S_{1,2} & \cdots & S_{1,2n+1} \\
S_{2,1} & S_{2,2} & \cdots & S_{2,2n+1} \\
\vdots & \vdots & \ddots & \vdots \\
S_{2n+1,1} & S_{2n+1,2} & \cdots & S_{2n+1,2n+1}
\end{bmatrix}$$
where the $S_{i,j}$'s are $2 \times 2$ matrices.
The condition that $J_{2n+1}(1) \otimes I_2$
is $S$-orthogonal is equivalent to having:
\begin{enumerate}[(i)]
\item $S_{i,j-1}+S_{i-1,j}+S_{i-1,j-1}=0$ for every $(i,j)\in \lcro 2,2n+1\rcro^2$;
\item $S_{k,1}=S_{1,k}=0$ for every $k\in \lcro 1,2n\rcro$.
\end{enumerate}
Set $K:=\begin{bmatrix}
0 & 1 \\
1 & 0
\end{bmatrix}$ and $L:=\begin{bmatrix}
0 & 0 \\
1 & 0
\end{bmatrix}$, and notice that $L+L^T=K$.
We define the $S_{i,j}$'s as follows:
\begin{itemize}
\item we set $S_{i,j}:=0$ whenever $i+j<2n+2$;
\item we set $S_{i,2n+2-i}:=K$ for every $i \in \lcro 1,2n+1\rcro$;
\item we set $S_{i,j}:=0$ whenever $i>n+1$, $j>n+1$ and $(i,j) \neq (2n+1,2n+1)$;
\item we set $S_{2n+1,2n+1}:=I_2$;
\item we set $S_{i,n+1}:=L$ whenever $i>n+1$;
\item we set $S_{n+1,j}:=L^T$ whenever $j>n+1$;
\item we then define (doubly)-inductively $S_{i,j}$ for $i$ from $n$ down to $2$ and for
$j$ from $2n-i+3$ up to $2n+1$ by $S_{i,j}:=-S_{i,j-1}-S_{i+1,j-1}$;
\item symmetrically, we define $S_{i,j}$ for $j$ from $n$ down to $2$ and for
$i$ from $2n-j+3$ up to $2n+1$ by $S_{i,j}:=-S_{i-1,j}-S_{i-1,j+1}$.
\end{itemize}
One checks that $S$ is symmetric, nonsingular, nonalternate (consider the last two diagonal entries)
and that $J_{2n+1}(1) \otimes I_2$ is $S$-orthogonal.
\end{proof}

\section{Refinements for quadratic forms over a finite field of characteristic 2}\label{car2quadSec}

In this section, $\K$ is a field of characteristic $2$.
If $\K$ is finite, there are two equivalence classes of regular quadratic forms of
dimension $2n$ over $\K$. We wish to determine, for a given essentially orthogonal automorphism $u \in \GL(\K^{2n})$,
the equivalence classes of the regular quadratic forms for which $u$ is orthogonal.

Since the theory of quadratic forms in characteristic $2$ is rather exotic, we start with a quick
reminder of some notations and basic facts.
\begin{itemize}
\item The map $\calP : x \mapsto x^2+x$ from $\K$ to $\K$ is a group homomorphism for $+$ with kernel $\{0,1\}$.
If $\K$ is finite, then its range is a subgroup of index $2$ of $\K$.
\item Given a regular quadratic form $q$ over a finite-dimensional $\K$-vector space, we choose a symplectic basis
of the polar form $b_q$: in this basis, $q$ is represented by a matrix of the form
$\begin{bmatrix}
A & I_n \\
0 & B
\end{bmatrix}$, and the class of $\tr(AB)$ in the quotient group $\K/\calP(\K) \simeq \Z/2$
is independent on the choices of the basis and of the matrices $A$ and $B$: this class, denoted by $\Delta(q)$,
is the \textbf{Arf invariant} of $q$. When $q$ is hyperbolic, its Arf invariant is $0$.
\item The Arf invariant is additive from $\bot$ to $+$, and when $\K$ is finite, it
classifies the regular quadratic forms of a given dimension up to equivalence.
\item In particular, for every $a \in \K$, the Arf invariant of the quadratic form
$[1,a]_\K : (x,y) \mapsto x^2+xy+ay^2$ on $\K^2$ is the class of $a$ in $\K/\calP(\K)$.
\end{itemize}

\subsection{A sufficient condition for being orthogonal for both types of regular quadratic forms}

Let $u \in \GL(V)$ be an essentially orthogonal automorphism.
We claim that if $1$ is an eigenvalue of $u$, then,
for any $\delta \in \K/\calP(\K)$, there is a regular quadratic form $q$ on $V$
with Arf invariant $\delta$ and for which $u$ is $q$-orthogonal.
Since Theorem \ref{symplectictheorem} shows that the odd-sized Jordan blocks of $u$ for the eigenvalue $1$ are paired,
it suffices to prove the following lemmas (the case of Jordan blocks of size one being trivial):

\begin{lemme}\label{evenjordancar2}
Let $n \in \N^*$ and $\delta \in \K/\calP(\K)$. Then there is a regular quadratic form on $\K^{2n}$
with Arf invariant $\delta$ for which $J_{2n}(1)$ is orthogonal.
\end{lemme}

\begin{lemme}\label{pairoddjordan2}
Let $n \in \N^*$ and $\delta \in \K/\calP(\K)$. Then there is a regular quadratic form on $\K^{4n+2}$
with Arf invariant $\delta$ for which $J_{2n+1}(1) \oplus J_{2n+1}(1)$ is orthogonal.
\end{lemme}

\begin{proof}[Proof of Lemma \ref{evenjordancar2}]
Denote by $(e_1,\dots,e_{2n})$ the canonical basis of $\K^{2n}$. \\
If $n=1$, we remark that any orthogonal group contains a reflection
and that every reflection of a $2$-dimensional vector space over $\K$ is represented by the matrix $J_2(1)$. \\
Assume now that $n \geq 2$, and let $a \in \K$. \\
We define $A=(a_{i,j}) \in \Mat_{2n}(\K)$ as follows:
\begin{itemize}
\item we set $a_{i,j}:=0$ whenever $i+j<2n+1$;
\item we set $a_{i,2n+1-i}:=1$ for every $i \in \lcro 1,2n\rcro$;
\item we set $a_{i,j}:=0$ whenever $i>n+1$ and $j>n+1$;
\item we set $a_{n+1,n+1}:=0$;
\item we set $a_{i,n+1}:=a$ whenever $i>n+1$, and  $a_{n+1,j}:=a$ whenever $j>n+1$;
\item we then define (doubly)-inductively $a_{i,j}$ for $i$ from $n$ down to $2$ and for
$j$ from $2n-i+2$ up to $2n$ by $a_{i,j}:=-a_{i,j-1}-a_{i+1,j-1}$;
\item symmetrically, we define $a_{i,j}$ for $j$ from $n$ down to $2$ and for
$i$ from $2n-j+2$ up to $2n$ by $a_{i,j}:=-a_{i-1,j}-a_{i-1,j+1}$.
\end{itemize}
The matrix $A$ is nonsingular and alternate, and $J_{2n}(1)^TAJ_{2n}(1)=A$.
Define then $q$ as the unique quadratic form on $\K^{2n}$ with polar form $b : (X,Y) \mapsto X^TAY$
and such that $q(e_{2n})=0$ and $q(e_i)=a_{i,i+1}$ for every $i \in \lcro 1,2n-1\rcro$.
Then $X \mapsto J_{2n}(1)X$ is $q$-orthogonal and $q(e_i)=0$ for every $i \in \lcro 1,n-1\rcro$. It follows that
$\Vect(e_1,\dots,e_{n-1})$ is totally $q$-isotropic. Using the hyperbolic inflation
principle (see \cite[Chapter VII, Proposition 3.2.5]{invitquad}) , we find that $q$ is Witt-equivalent
(see \cite[Chapter IX, Definition 1.0.26]{invitquad})
to its restriction $q'$ to $\Vect(e_n,e_{n+1})$
(notice that $\Vect(e_1,\dots,e_{n+1})$ is the orthogonal of $\Vect(e_1,\dots,e_{n-1})$ for $b$).
However $q(e_n)=b(e_n,e_{n+1})=1$ and $q(e_{n+1})=b(e_{n+1},e_{n+2})=a$, and
$(e_n,e_{n+1})$ is a symplectic basis of $\Vect(e_n,e_{n+1})$, hence
$q' \simeq [1,a]_\K$. We deduce that $\Delta(q)=[a]$, which completes the proof.
\end{proof}

\begin{proof}[Proof of Lemma \ref{pairoddjordan2}]
The strategy is quite similar to that of our proof of Lemma \ref{pairoddjordan1}.
We work with $M:=J_{2n+1}(1) \otimes I_2$.
We let $a \in \K$ and we find a nonsingular alternate matrix $S$ of the form
$$\begin{bmatrix}
S_{1,1} & S_{1,2} & \cdots & S_{1,2n+1} \\
S_{2,1} & S_{2,2} & \cdots & S_{2,2n+1} \\
\vdots & \vdots & \ddots & \vdots \\
S_{2n+1,1} & S_{2n+1,2} & \cdots & S_{2n+1,2n+1}
\end{bmatrix}$$
such that $M^TSM=S$, where the $S_{i,j}$'s are $2 \times 2$ matrices.
Set $K:=\begin{bmatrix}
0 & 1 \\
1 & 0
\end{bmatrix}$ and $L:=\begin{bmatrix}
a & 0 \\
1 & 1
\end{bmatrix}$, and notice that $L+L^T=K$.
We then define the $S_{i,j}$'s as follows:
\begin{itemize}
\item we set $S_{i,j}:=0$ whenever $i+j<2n+2$;
\item we set $S_{i,2n+2-i}:=K$ for every $i \in \lcro 1,2n+1\rcro$;
\item we set $S_{i,j}:=0$ whenever $i>n+1$ and $j>n+1$;
\item we set $S_{i,n+1}:=L$ whenever $i>n+1$;
\item we set $S_{n+1,j}:=L^T$ whenever $j>n+1$;
\item we then define (doubly)-inductively $S_{i,j}$ for $i$ from $n$ down to $2$ and for
$j$ from $2n-i+3$ up to $2n+1$ by $S_{i,j}:=-S_{i,j-1}-S_{i+1,j-1}$;
\item symmetrically, we define $S_{i,j}$ for $j$ from $n$ down to $2$ and for
$i$ from $2n-j+3$ up to $2n+1$ by $S_{i,j}:=-S_{i-1,j}-S_{i-1,j+1}$.
\end{itemize}
One checks that $S$ is alternate and nonsingular, and that $M^TSM=S$.
Denote by $(e_1,\dots,e_{4n+2})$ the canonical basis of $\K^{4n+2}$,
and let $q$ be an arbitrary quadratic form on $\K^{4n+2}$ with polar form $b : (X,Y) \mapsto X^TSY$.
Then $M$ is $q$-orthogonal if and only if $q(e_k)=q(e_k+e_{k-2})$ for every $k \in \lcro 3,4n+2\rcro$, i.e.,
$q(e_k)=b(e_k,e_{k+2})$ for every $k \in \lcro 1,4n\rcro$.

In this case, $q(e_1)=q(e_2)=\cdots=q(e_{2n})=0$, and since
$\Vect(e_1,\dots,e_{2n})$ is totally $b$-singular, it follows that it is also totally $q$-isotropic.
As in the previous proof, the hyperbolic inflation principle ensures that $q$ is Witt-equivalent to
its restriction $q'$ to $\Vect(e_{2n+1},e_{2n+2})$. However $q(e_{2n+1})=b(e_{2n+1},e_{2n+3})=a$
and $q(e_{2n+2})=b(e_{2n+2},e_{2n+4})=1$, and $(e_{2n+1},e_{2n+2})$ is a $b$-symplectic basis of
 $\Vect(e_{2n+1},e_{2n+2})$, hence $\Delta(q')=[a]$. We deduce that $\Delta(q)=[a]$, which completes the proof.
\end{proof}

\begin{cor}Assume that $\K$ is finite and
let $u \in \GL(V)$ be an essentially symplectic morphism. If $1$ is an eigenvalue of $u$,
then there exists an hyperbolic form on $V$ for which $u$ is orthogonal
and a regular non-hyperbolic quadratic form on $V$ for which $u$ is orthogonal.
\end{cor}

\subsection{On essentially symplectic morphisms for which $1$ is not an eigenvalue}

In this section, we assume that $\K$ is finite and we
consider an essentially symplectic morphism $u$ of which $1$ is not an eigenvalue.
We intend to prove that all the quadratic forms for which $u$ is orthogonal are equivalent: we do so by calculating their Arf invariant.
Let $q$ be a regular quadratic form for which $u$ is orthogonal
(we say that $q$ is \textbf{$u$-adapted}), and denote by $b_q$ its polar form.

\begin{lemme}\label{totallyisotropic}
Let $W$ be a linear subspace of $V$ which is both totally $b_q$-singular
and stabilized by $u$. Then $q$ vanishes on $W$.
\end{lemme}

\begin{proof}
Indeed, for every $x \in W$, one has
$q((u-\id)(x))=q(u(x))+q(x)+b_q(x,u(x))=0$, and the result follows since $u-\id$ is an automorphism of the
finite-dimensional vector space $W$.
\end{proof}

We now split up the minimal polynomial of $u$ as
$$\mu_u=Q\,Q^\#\,P_1^{a_1}\cdots P_p^{a_p},$$
where $Q$ is prime with $Q^\#$, and $P_1,\dots,P_p$ are pairwise distinct irreducible palindromials with a degree greater than $1$
(and the $a_i$'s are positive integers).
Both $Q$ and $Q^\#$ are prime with each $P_k$.
The subspaces $W:=\Ker(Q\,Q^\#)(u), V_1:=\Ker P_1^{a_1}(u),\dots,V_p:=\Ker P_p^{a_p}(u)$ must then be pairwise $q$-orthogonal, hence
$$q \simeq q_W \bot q_{V_1} \bot \dots \bot q_{V_p.}$$
Notice that $q_W$ is hyperbolic: indeed $\Ker Q(u)$ and $\Ker Q^\#(u)$ are both totally $b_q$-singular
and stabilized by $u$, so the previous lemma shows that $q$ vanishes on both of them.
It follows that
$$\Delta(q)=\underset{k=1}{\overset{p}{\sum}} \Delta(q_{V_k}).$$
It now suffices to investigate the case $\mu_u$ is a power of an irreducible palindromial $P$ with $\deg P >1$:

\begin{prop}\label{onedivisorcar2}
Let $P \in \K[x]$ be an irreducible palindromial of degree greater than $1$.
Let $u \in \GL(V)$ be an automorphism whose minimal polynomial is a power of $P$, and let $q$
be a regular $u$-adapted quadratic form on $V$.
Denote by $N$ the number of blocks of type $C(P^{2a+1})$ (with $a \in \N$) in the primary canonical form of
$u$. Then $\Delta(q)=N.[\varepsilon]$, where $\varepsilon \in \K \setminus \calP(\K)$.
\end{prop}

Before proving this result, we immediately deduce our final theorem:

\begin{theo}\label{subtilquadformscar2}
Assume that $\K$ is a finite field of characteristic $2$.
Let $u \in \GL(V)$ be an essentially symplectic automorphism,
and denote by $N$ the number of blocks of type $C(P^{2a+1})$, with $a \in \N$ and $P$ an irreducible palindromial of degree greater than $1$,
in the primary canonical form of $u$.
\begin{enumerate}[(a)]
\item If $1$ is an eigenvalue of $u$, then there is an hyperbolic form $q_1$ on $V$ and a regular non-hyperbolic form $q_2$ on $V$
such that $u \in \Ortho(q_1) \cap \Ortho(q_2)$.
\item If $1$ is not an eigenvalue of $u$ and $N$ is even, then every $u$-adapted regular quadratic form on $V$
is hyperbolic.
\item If $1$ is not an eigenvalue of $u$ and $N$ is odd, then every $u$-adapted regular quadratic form on $V$ is non-hyperbolic.
\end{enumerate}
\end{theo}

We now turn to the proof of Proposition \ref{onedivisorcar2}.
To start with, we use essentially the same method as in the proof of Proposition \ref{decortiquage}
(see Section \ref{decorsection}). With the same notations,
we replace the endomorphism $u-\id$ with $P(u)$ and add the condition that the subspace
$\underset{i=p}{\overset{k}{\bigoplus}}V_{k,i}$ be stabilized by $u$ for every
$k \in \lcro 1,2n\rcro$ and every $p \in \lcro 1,k\rcro$. Finally, for every $x \in F \oplus G \oplus H$,
$v(x)$ is defined as the unique vector of $F \oplus G \oplus H$ such that $v(x)-u(x) \in E$
(that we may decompose $V$ into the sum of the $V_{k,i}$'s is a classical consequence of the generalized Jordan reduction
theorem).

Notice that $\Ker P(v)=F\oplus H$ and $\im P(v)=H$.
The rest of the arguments of Section \ref{decorsection} hold in this new context,
which shows that $E$ is totally $b_q$-singular, hence totally $q$-isotropic by Lemma \ref{totallyisotropic}.
Applying again Lemma \ref{totallyisotropic} to $v$ on $H$, we find that $H$ is totally $q$-isotropic.
The hyperbolic inflation theorem then ensures that $q$ is Witt-equivalent to $q_F$,
which leads to $\Delta(q)=\Delta(q_F)$.

We have thus reduced the situation to the one where $P$ is the minimal polynomial of $u$, which we now consider.
As in Section \ref{finitecompanion}, we set $\L:=\K[x]/(P(x))$, denote by $y$ the class of the indeterminate $x$ in $\L$,
by $\sigma$ the $\K$-automorphism of $\L$ such that $\sigma(y)=y^{-1}$,
and we set $\K':=\{z \in \L : \; \sigma(z)=z\}$. Notice that $u$ induces a structure of $\L$-vector space on $V$.
This reduces the situation to the one where $V=\L^n$ for some $n \geq 1$, and $u$ is the multiplication by $y$ in the vector space $\L^n$.

\begin{lemme}\label{reprbilin}
Let $B$ be a symmetric bilinear form on the $\K$-vector space $\L$ such that $B(ya,yb)=B(a,b)$ for every $(a,b)\in \L^2$.
Then there is a (unique) $c \in \L$ such that $B(a,b)=\Tr_{\L/\K}(c\,\sigma(a)b)$ for every $(a,b)\in \L^2$.
\end{lemme}

\begin{proof}
Since $(a,b) \mapsto \Tr_{\L/\K}(ab)$ is a regular bilinear form on the $\K$-vector space $\L$, there is a unique
endomorphism $\varphi$ of the $\K$-vector space $\L$ such that $\forall (a,b) \in \L^2, \; B(a,b)=
\Tr_{\L/\K}(\varphi(a)b)$. Since $\K$ is a finite field, $\L$ is a Galois extension of $\K$. Therefore, we may decompose
$\varphi=\underset{\tau \in \Gal(\L/\K)}{\sum} \lambda_\tau.\tau$ for a unique family $(\lambda_\tau)$ of elements of $\L$.
However, $B(ya,yb)=B(a,b)$ for every $(a,b)\in \L^2$ and hence
the uniqueness of $\varphi$ shows that $\varphi(yz)y=\varphi(z)$ for every $z \in \L$. Since
$y\varphi(yz)=\underset{\tau \in \Gal(\L/\K)}{\sum} y\tau(y)\lambda_\tau.\tau(z)$ for every $z \in \L$, we deduce that
 $y\tau(y)\lambda_\tau=\lambda_\tau$ for every $\tau \in \Gal(\L/\K)$.
Let $\tau \in \Gal(\L/\K)$. If $\tau(y)=y^{-1}$, then $\tau=\sigma$ since $y$ generates $\L$ as a $\K$-algebra.
We deduce that $\lambda_\tau=0$ whenever $\tau \neq \sigma$. Therefore
$\varphi=\lambda_\sigma\, \sigma$, hence $c:=\lambda_\sigma$ has the required properties.
\end{proof}

\begin{claim}
The polar form $b_q$ of $q$ has the form $(X,Y) \mapsto \Tr_{\L/\K}\bigl(\sigma(X)^TAY\bigr)$ for some nonsingular
hermitian matrix $A \in \Mat_n(\L)$ (hermitian in the sense that $\sigma(A)^T=A$).
\end{claim}

\begin{proof}
Denote by $(e_1,\dots,e_n)$ the canonical basis of the $\L$-vector space $\L^n$. For every $(i,j)\in \lcro 1,n\rcro^2$,
the map $b_{i,j} : (a,b) \mapsto b_q(a e_i,b e_j)$ satisfies the conditions of Lemma \ref{reprbilin}, so we may
find $z_{i,j} \in \L$ such that $b_{i,j}(a,b)=\Tr_{\L/\K}(z_{i,j}\sigma(a)b)$ for every $(a,b)\in \L^2$.
Set then $A:=(z_{i,j})_{1 \leq i,j \leq n} \in \Mat_n(\L)$ and notice that
$$\forall (X,Y)\in (\L^n)^2, \; b_q(X,Y)=\Tr_{\L/\K}(\sigma(X)^T A Y).$$
Since $b_q$ is symmetric, it follows that
$$\Tr_{\L/\K}(\sigma(X)^T A Y)=\Tr_{\L/\K}(\sigma(Y)^T A X)=\Tr_{\L/\K}(\sigma(\sigma(Y)^T A X)^T)
=\Tr_{\L/\K}(\sigma(X)^T \sigma(A)^T Y)$$
for every $(X,Y) \in (\L^n)^2$.
Let $(X,Y)\in (\L^n)^2$ and $\lambda \in \L$. Applying the previous identity to $(X,\lambda Y)$
yields $\Tr_{\L/\K}\bigl(\lambda(\sigma(X)^T A Y-\sigma(X)^T \sigma(A)^T Y)\bigr)=0$.
It follows that $\sigma(X)^T A Y-\sigma(X)^T \sigma(A)^T Y=0$. Therefore $\sigma(A)^T=A$.
Finally, since $b_q$ is regular, we find that $A$ is nonsingular.
\end{proof}

The relations $\forall x \in V, \; q(u(x)-x)=b_q(x,u(x))$ and the fact that $u-\id$ is an automorphism of $V$
show that $q$ is the unique quadratic form on $V$ with polar form $b_q$ such that $u \in \text{O}(q)$.
Since the map $X \mapsto \Tr_{\K'/\K}\bigl(\sigma(X)^TAX\bigr)$ qualifies, it equals $q$.

Notice that the Gaussian reduction of hermitian forms still holds in characteristic $2$. Since $\K$ is perfect,
we deduce that $q$ is equivalent to the form $(x_1,\dots,x_n) \mapsto
\Tr_{\K'/\K}\bigl(x_1\sigma(x_1)+\cdots+x_n\sigma(x_n)\bigr)$, which is itself equivalent to the orthogonal sum of
$n$ copies of the form $x \mapsto \Tr_{\K'/\K}(x\sigma(x))$ on the $\K$-vector space $\L$.
In order to conclude, we prove the following:

\begin{claim}\label{lastclaim}
The quadratic form $\varphi : x \mapsto \Tr_{\K'/\K}(x\sigma(x))$ on the $\K$-vector space $\L$ is non-hyperbolic.
\end{claim}

Choose $a \in \K' \setminus \calP(\K')$. Notice that $\varphi : x \mapsto x\sigma(x)$ is a regular non-isotropic
quadratic form on the $2$-dimensional $\K'$-vector space $\L$, hence its Arf invariant is $a$.
This shows that this form is equivalent to $[a,1]_{\K'}$ (both have the same Arf invariant).
Let $\bfB:=(e_1,\dots,e_m)$ be a basis of the $\K$-vector space $\K'$ and set
$P:=(\Tr_{\K'/\K}(e_ie_j))_{1 \leq i,j \leq m}$ and $P_a:=(\Tr_{\K'/\K}(ae_ie_j))_{1 \leq i,j \leq m}$.
Obviously, the matrix $\begin{bmatrix}
P_a & P \\
0 & P
\end{bmatrix}$ represents $\varphi$ in some basis of the $\K$-vector space $\L$.
Multiplying it by $C:=\begin{bmatrix}
I_m & 0 \\
0 & P^{-1}
\end{bmatrix}$ on the left and by $C^T$ on the right, we find that
$\begin{bmatrix}
P_a & I_m \\
0 & P^{-1}
\end{bmatrix}$ represents $\varphi$, hence $\Delta(\varphi)=\bigl[\tr(P_aP^{-1})]$.
However $\tr(P_aP^{-1})=\tr(P^{-1}P_a)=\Tr_{\K'/\K}(a)$ since $P_a=P \times \Mat_{\bfB}(x \mapsto ax)$.
In order to conclude the proof of Claim \ref{lastclaim}, it suffices
to establish the following final lemma:

\begin{lemme}
Let $\K-\L$ be an extension of finite fields of characteristic $2$.
Then $\Tr_{\L/\K}$ induces a group isomorphism from $\L/\calP(\L)$ to $\K/\calP(\K)$.
\end{lemme}

\begin{proof}
For every $x \in \L$, notice that
$$\Tr_{\L/\K}(x^2+x)=\sum_{\sigma \in \Gal(\L/\K)}\sigma(x^2+x)=\sum_{\sigma \in \Gal(\L/\K)} \sigma(x)^2
+\sum_{\sigma \in \Gal(\L/\K)} \sigma(x)=\calP(\Tr_{\L/\K}(x)),$$
therefore $\Tr_{\L/\K}$ maps $\calP(\L)$ into $\calP(\K)$. It follows that
$\Tr_{\L/\K}$ induces a group homomorphism from $\L/\calP(\L)$ to $\K/\calP(\K)$.
This homomorphism is onto since $\Tr_{\L/\K}$ maps $\L$ onto $\K$, being a non-zero $\K$-linear form on $\L$.
Since both groups $\L/\calP(\L)$ and $\K/\calP(\K)$ have order $2$, the claimed result follows.
\end{proof}

This finishes the proof of Proposition \ref{onedivisorcar2} and Theorem \ref{subtilquadformscar2}.

\end{document}